\documentclass[oneside,english]{amsart}
\usepackage[T1]{fontenc}
\usepackage[latin9]{inputenc}
\usepackage{amsthm}
\usepackage{amstext}
\usepackage{amssymb}
\usepackage{esint}
\usepackage{graphicx}
\usepackage{enumerate}

\usepackage{graphicx}
\graphicspath{{noiseimages/}}

\makeatletter
\newtheorem{theorem}{Theorem}[section]
\newtheorem{lemma}[theorem]{Lemma}

\newtheorem{corollary}[theorem]{Corollary}
\numberwithin{figure}{section}
\theoremstyle{definition}

\theoremstyle{remark}
\newtheorem{remark}[theorem]{Remark}

\numberwithin{equation}{section}
\makeatother

\usepackage{babel}

	\DeclareMathOperator{\loc}{loc}

	\DeclareMathOperator{\spn}{span}
	\DeclareMathOperator{\Imag}{Im}

\begin{document}

\title[Spectral Stability]{Spectral Stability Estimates of Dirichlet Divergence Form Elliptic Operators}

\author{V.~Gol'dshtein, V.~Pchelintsev, A.~Ukhlov}

\begin{abstract}
We study spectral stability estimates of elliptic operators in divergence form $-\textrm{div} [A(w) \nabla g(w)]$ with the Dirichlet boundary condition in non-Lipschitz domains $\widetilde{\Omega} \subset \mathbb C$. The suggested method is based on connections of planar quasiconformal mappings with Sobolev spaces and its applications to the Poincar\'e inequalities.
\end{abstract}
\maketitle

\section{Introduction}

In this paper we give applications of quasiconformal mappings to spectral stability estimates of the Dirichlet eigenvalues of $A$-divergent form elliptic operators:
\begin{equation}\label{EllDivOper}
L_{A}=-\textrm{div} [A(w) \nabla g(w)], \quad w=(u,v)\in \widetilde{\Omega}, \quad g|_{\partial \widetilde{\Omega}}=0,
\end{equation}
in non-Lipschitz domains $\widetilde{\Omega} \subset \mathbb C$ with $A \in M^{2 \times 2}(\widetilde{\Omega})$. We denote, by $M^{2 \times 2}(\widetilde{\Omega})$, the class of all $2 \times 2$ symmetric matrix functions $A(w)=\left\{a_{kl}(w)\right\}$, $\textrm{det} A=1$, with measurable entries satisfying the uniform ellipticity condition
\begin{equation}\label{UEC}
\frac{1}{K}|\xi|^2 \leq \left\langle A(w) \xi, \xi \right\rangle \leq K |\xi|^2 \,\,\, \text{a.~e. in}\,\,\, \widetilde{\Omega},
\end{equation}
for every $\xi \in \mathbb C$, where $1\leq K< \infty$.
Such type elliptic operators in divergence form arise in various problems of mathematical physics (see, for example, \cite{AIM}).

In the general case quasiconformal mappings are Sobolev mappings and we will consider this eigenvalue problem in the weak formulation:
\begin{equation}\label{WFEP}
\iint\limits_{\widetilde{\Omega}} \left\langle \nabla g(w), \nabla \overline{f(w)} \right\rangle dudv
= \lambda \iint\limits_{\widetilde{\Omega}} g(w)\overline{f(w)}~dudv, \,\,\, \forall f\in W_0^{1,2}(\widetilde{\Omega}). 
\end{equation} 

It is known \cite{Henr,M} that in bounded domains $\widetilde{\Omega} \subset \mathbb C$ the  Dirichlet spectrum of the divergence form elliptic operators $-\textrm{div} [A(w) \nabla g(w)]$ is discrete and can be written in the form of a non-decreasing sequence
\[
0<\lambda_1[A,\widetilde{\Omega}] \leq \lambda_2[A,\widetilde{\Omega}] \leq \ldots \leq \lambda_n[A,\widetilde{\Omega}] \leq \ldots,
\] 
where each eigenvalue is repeated as many time as its multiplicity. 

The suggested approach to the spectral stability is based on connection between $A$-divergent form elliptic operators and quasiconformal mappings. Let us consider a domain $\Omega_e$ which is an interior of the ellipse
$$
(u,v) \in \mathbb R^2: \frac{u^2}{a^2}+\frac{v^2}{b^2}=1,
$$
where $a=\sqrt{c^2+1}+c$, $b=\sqrt{c^2+1}-c$, $c\geq 0$. Then the quasiconformal mapping  
\[
\varphi(w)= \sqrt{c^2+1}w-c \overline{w}, \quad w=u+iv,
\]
maps $\Omega_e$ onto the unit disc $\mathbb D$, and reduces by the composition rule $g=f\circ\varphi$, the divergent form spectral problem in $\Omega_e$
$$ 
-\textrm{div} [A(w) \nabla g(w)]=\lambda g, \quad g|_{\partial \Omega_e}=0,
$$
with the matrix  
$$
A(w)=\begin{pmatrix} (\sqrt{c^2+1}+c)^2 & 0 \\ 0 &  (\sqrt{c^2+1}-c)^2 \end{pmatrix},
$$ 
to the spectral problem for the Laplace operator in the unit disc $\mathbb D$
$$ 
-\Delta f=\lambda f, \quad f|_{\partial \mathbb D}=0.
$$
Hence (the detailed calculations will be given in section 6) the Dirichlet eigenvalues of the divergent form spectral problem 
$$
\lambda_{m,n}[A,\Omega_e]=\lambda_{m,n}[\mathbb D]=j^2_{m,n},\,\,\,m=0,1,2,\ldots,\,\,\, n=1,2,\ldots,
$$
where $j_{m,n}$ is the $n$-th zero of the $m$-th Bessel function $J_m$, i.e. $J_m(j_{m,n})=0$.

Hence quasiconformal type composition operators allow study spectral problems of elliptic operators. In the present paper we obtain spectral stability estimates in the case of the Dirichlet eigenvalues of the divergence form elliptic operators in non-Lipschitz domains. 
The suggested method is based on connections between the quasiconformal mappings agreed with the matrix $A$ \cite{AIM, GNR18} and composition operators on Sobolev spaces \cite{GPU2019}. The main results of the article concern to spectral stability estimates in domains that we call as $A$-quasiconformal $\beta$-regularity domains. Namely a simply connected domain $\widetilde{\Omega} \subset \mathbb C$ is called an $A$-quasiconformal $\beta$-regular domain about a simply connected domain ${\Omega} \subset \mathbb C$ if
$$
\iint\limits_{\widetilde{\Omega}} |J(w, \varphi)|^{1-\beta}~dudv < \infty, \,\,\,\beta>1,
$$
where $J(w, \varphi)$ is a Jacobian of an $A$-quasiconformal mapping $\varphi: \widetilde{\Omega}\to\Omega$. Note that $A$-quasiconformal $\beta$-regular domains satisfy the quasihyperbolic boundary condition \cite{KOT} and has finite geodesic diameter \cite{GU14}. An important subclass of $A$-quasiconformal regular domains are Ahlfors domains \cite{Ahl66}.  

The spectral stability estimates of the self-adjoint elliptic operators were intensively studied during the last decade. See, for example, \cite{BBL, BL1, BL2, BL3, BLLdeC, LP, LMS} where the history of the problem and main results in this area can be found. 
In  the previous works \cite{BGU15, BGU16}, using an approach which is based on the conformal theory of composition operators on Sobolev spaces were established the conformal spectral stability estimates of Dirichlet eigenvalues and Neumann eigenvalues of Laplacian in non-Lipschitz domains that include some fractal type domains like snowflakes.    

The main result of the article states that if a domain $\widetilde{\Omega}$ is $A$-quasiconformal $\beta$-regular about $\Omega$, then for any $n\in \mathbb N$ the following spectral stability estimates hold:
\begin{multline*}
|\lambda_n[A, \Omega]-\lambda_n[A, \widetilde{\Omega}]| 
\leq c_n A^2_{\frac{4\beta}{\beta -1},2}(\Omega) \\
\times \left(|\Omega|^{\frac{1}{2\beta}} +
\|J_{\varphi^{-1}}\,|\,L^{\beta}(\Omega)\|^{\frac{1}{2}} \right) \cdot
\|1-J_{\varphi^{-1}}^{\frac{1}{2}}\,|\,L^{2}(\Omega)\|,
\end{multline*}
where
$c_n=\max\left\{\lambda_n^2[A, \Omega], \lambda_n^2[A, \widetilde{\Omega}]\right\}$,  $J_{\varphi^{-1}}$ is a Jacobian of an $A^{-1}$-quasiconformal mapping $\varphi^{-1}:\Omega\to\widetilde{\Omega}$, and
\[
A_{\frac{4\beta}{\beta -1},2}(\Omega) \leq \inf\limits_{p\in \left(\frac{4\beta}{3\beta -1},2\right)} 
\left(\frac{p-1}{2-p}\right)^{\frac{p-1}{p}}
\frac{\left(\sqrt{\pi}\cdot\sqrt[p]{2}\right)^{-1}|\Omega|^{\frac{\beta-1}{4\beta}}}{\sqrt{\Gamma(2/p) \Gamma(3-2/p)}}~~.
\]

In case of quasiconformal diffeomorphisms $\varphi:\widetilde{\Omega} \to \Omega$ for which $|J(w,\varphi)|=1$ 
the operator \eqref{EllDivOper} can be rewritten as
\begin{equation}\label{quasiconf-eq}
-\textrm{div} [A(w) \nabla g(w)] =-\Delta f(z).
\end{equation}
This means that 
the Dirichlet eigenvalues of the divergence form elliptic operator $-\textrm{div} [A(w) \nabla g(w)]$ in a domain $\widetilde{\Omega}$ are equal to the Dirichlet eigenvalues of the Laplace operator in a domain $\Omega$. Note that such type operators are called ``\,isospectral''. 
Examples of so-called isospectral operators will be given in Section 6. 

The method proposed to investigation the weak weight eigenvalue problem for the Dirichlet Laplacian is based on the Sobolev embedding theorems \cite{GG94,GU09} in connection with the composition operators on Sobolev spaces \cite{U93, VU02}.

\section{Sobolev spaces and $A$-quasiconformal mappings}

Let $E \subset \mathbb C$ be a measurable set on the complex plane and $h:E \to \mathbb R$ be a positive a.e. locally integrable function i.e. a weight. The weighted Lebesgue space $L^p(E,h)$, $1\leq p<\infty$, 
is the space of all locally integrable functions with the finite norm
$$
\|f\,|\,L^{p}(E,h)\|= \left(\iint\limits_E|f(z)|^ph(z)\,dxdy \right)^{\frac{1}{p}}< \infty.
$$

The two-weighted Sobolev space $W^{1,p}(\Omega,h,1)$, $1\leq p< \infty$, is defined
as the normed space of all locally integrable weakly differentiable functions
$f:\Omega\to\mathbb{R}$ endowed with the following norm:
\[
\|f\mid W^{1,p}(\Omega,h,1)\|=\|f\,|\,L^{p}(\Omega,h)\|+\|\nabla f\mid L^{p}(\Omega)\|.
\]

In the case $h=1$ this weighted Sobolev space coincides with the classical Sobolev space $W^{1,p}(\Omega)$. 
The seminormed Sobolev space $L^{1,p}(\Omega)$, $1\leq p< \infty$,
is the space of all locally integrable weakly differentiable functions $f:\Omega\to\mathbb{R}$ endowed
with the following seminorm: 
\[
\|f\mid L^{1,p}(\Omega)\|=\|\nabla f\mid L^p(\Omega)\|, \,\, 1\leq p<\infty.
\]

We also need a weighted seminormed Sobolev space $L^{1,2}(\Omega, A)$ (associated with the matrix $A$), defined 
as the space of all locally integrable weakly differentiable functions $f:\Omega\to\mathbb{R}$
with the finite seminorm given by:  
\[
\|f\mid L^{1,2}(\Omega,A)\|=\left(\iint\limits_\Omega \left\langle A(z)\nabla f(z),\nabla f(z)\right\rangle\,dxdy \right)^{\frac{1}{2}}.
\]

The corresponding  Sobolev space $W^{1,2}(\Omega, A)$ is defined
as the normed space of all locally integrable weakly differentiable functions
$f:\Omega\to\mathbb{R}$ endowed with the following norm:
\[
\|f\mid W^{1,2}(\Omega, A)\|=\|f\,|\,L^{2}(\Omega)\|+\|f\mid L^{1,2}(\Omega,A)\|.
\]
The Sobolev space $W^{1,2}_{0}(\Omega, A)$ is the closure in the $W^{1,2}(\Omega, A)$-norm of the 
space $C^{\infty}_{0}(\Omega)$.

Recall that a homeomorphism $\varphi: \widetilde{\Omega}\to \Omega$, $\widetilde{\Omega},\, \Omega \subset\mathbb C$, is called a $K$-quasiconformal mapping if $\varphi\in W^{1,2}_{\loc}(\widetilde{\Omega})$ and there exists a constant $1\leq K<\infty$ such that
$$
|D\varphi(w)|^2\leq K |J(w,\varphi)|\,\,\text{for almost all}\,\,w \in \widetilde{\Omega}.
$$

Now we give a construction of $A$-quasiconformal mappings connected with the $A$-divergent form elliptic operators.
The basic idea is that every positive quadratic form
\[
ds^2=a_{11}(u,v)du^2+2a_{12}(u,v)dudv+a_{22}(u,v)dv^2
\]
defined in a planar domain $\widetilde{\Omega}$ can be reduced, by means of a quasiconformal change of variables, to the canonical form
\[
ds^2=\Lambda (dx^2+dy^2),\,\, \Lambda \neq 0,\,\, \text{a.e. in}\,\, \Omega,
\]
given that $a_{11}a_{22}-a^2_{12} \geq \kappa_0>0$, $a_{11}>0$, almost everywhere in $\widetilde{\Omega}$ \cite{Ahl66, BGMR}. Note that this fact can be extended to linear operators of the form $\textrm{div} [A(w) \nabla g(w)]$, $w=u+iv$, for matrix function $A \in M^{2 \times 2}(\widetilde{\Omega})$.

Let $\xi(w)=\Re(\varphi(w))$ be a real part of a quasiconformal mapping $\varphi(w)=\xi(w)+i \eta(w)$, which satisfies to the Beltrami equation:
\begin{equation}\label{BelEq}
\varphi_{\overline{w}}(w)=\mu(w) \varphi_{w}(w),\,\,\, \text{a.e. in}\,\,\, \widetilde{\Omega},
\end{equation}
where
$$
\varphi_{w}=\frac{1}{2}\left(\frac{\partial \varphi}{\partial u}-i\frac{\partial \varphi}{\partial v}\right) \quad \text{and} \quad 
\varphi_{\overline{w}}=\frac{1}{2}\left(\frac{\partial \varphi}{\partial u}+i\frac{\partial \varphi}{\partial v}\right),
$$
with the complex dilatation $\mu(w)$ is given by
\begin{equation}\label{ComDil}
\mu(w)=\frac{a_{22}(w)-a_{11}(w)-2ia_{12}(w)}{\det(I+A(w))},\quad I= \begin{pmatrix} 1 & 0 \\ 0 & 1 \end{pmatrix}.
\end{equation}

Then the uniform ellipticity condition \eqref{UEC} can be written as
\begin{equation}\label{OVCE}
|\mu(w)|\leq \frac{K-1}{K+1},\,\,\, \text{a.e. in}\,\,\, \widetilde{\Omega}.
\end{equation}

Conversely we can obtain from \eqref{ComDil} (see, for example, \cite{AIM}, p. 412) that :
\begin{equation}\label{Matrix-F}
A(w)= \begin{pmatrix} \frac{|1-\mu|^2}{1-|\mu|^2} & \frac{-2 \Imag \mu}{1-|\mu|^2} \\ \frac{-2 \Imag \mu}{1-|\mu|^2} &  \frac{|1+\mu|^2}{1-|\mu|^2} \end{pmatrix},\,\,\, \text{a.e. in}\,\,\, \widetilde{\Omega}.
\end{equation}

So, given any $A \in M^{2 \times 2}(\widetilde{\Omega})$, one produced, by \eqref{OVCE}, the complex dilatation $\mu(w)$, for which, in turn, the Beltrami equation \eqref{BelEq} induces a quasiconformal homeomorphism $\varphi:\widetilde{\Omega} \to \Omega$ as its solution, by the Riemann measurable mapping theorem (see, for example, \cite{Ahl66}). We will say that the matrix function $A$ induces the corresponding $A$-quasiconformal homeomorphism $\varphi$ or that $A$ and $\varphi$ are agreed. 

So, by the given $A$-divergent form elliptic operator defined in a domain $\widetilde{\Omega}\subset\mathbb C$ we construct so-called a $A$-quasiconformal mapping $\varphi:\widetilde{\Omega}\to\Omega$ with a quasiconformal coefficient
$$
K=\frac{1+\|\mu\mid L^{\infty}(\widetilde{\Omega})\|}{1-\|\mu\mid L^{\infty}(\widetilde{\Omega})\|},
$$
where $\mu$ defined by (\ref{ComDil}).

Note that the inverse mapping to the $A$-quasiconformal mapping $\varphi: \widetilde{\Omega} \to \Omega$ is the $A^{-1}$-quasiconformal mapping \cite{GPU2019}.

In \cite{GPU2019} was given a connection between composition operators on Sobolev spaces and $A$-quasiconformal mappings. 

\begin{theorem} \label{L4.1}
Let $\Omega,\widetilde{\Omega}$ be domains in $\mathbb C$. Then a homeomorphism $\varphi :\widetilde{\Omega} \to \Omega$ is an $A$-quasiconformal mapping 
if and only if $\varphi$ induces, by the composition rule $\varphi^{*}(f)=f \circ \varphi$,
an isometry of Sobolev spaces $L^{1,2}(\widetilde{\Omega},A)$ and $L^{1,2}(\Omega)$:
\[
\|\varphi^{*}(f)\,|\,L^{1,2}(\widetilde{\Omega},A)\|=\|f\,|\,L^{1,2}(\Omega)\|
\]
for any $f \in L^{1,2}(\Omega)$.
\end{theorem}

This theorem generalizes the well known property of conformal mappings generate the isometry of uniform Sobolev spaces $L^1_2(\Omega)$ and $L^1_2(\widetilde{\Omega})$ (see, for example, \cite{C50}) and refines (in the case $n=2$) the functional characterization of quasiconformal mappings in the terms of isomorphisms of uniform Sobolev spaces \cite{VG75}.

\section{Weighted eigenvalue problems}

In this section we estimate variation of Dirichlet eigenvalues of two weighted spectral problems defined in a domain $\Omega\subset\mathbb R^n$. Let $w_1$, $w_2$ be positive a.e. locally integrable functions (i.e. weights) defined in a domain $\Omega$. We assume that weighted embedding operators
\[
i_{\Omega}:W_0^{1,2}(\Omega) \hookrightarrow L^2(\Omega, w_k),\,\, k=1,2,
\]
are compact (see, for example, \cite{M}). Then weighted eigenvalue problems
\[
\iint\limits _{\Omega} \left\langle \nabla f(z), \nabla\overline{{g(z)}}\right\rangle~dxdy=\lambda\iint\limits _{\Omega}w_{1}(z)f(z)\overline{g(z)}~dxdy\,,\,\,~~\forall g\in W_{0}^{1,2}(\Omega)
\]
 and
\[
\iint\limits _{\Omega} \left\langle \nabla f(z), \nabla\overline{{g(z)}}\right\rangle~dxdy=\lambda\iint\limits _{\Omega}w_{2}(z)f(z)\overline{{g(z)}}~dxdy\,,\,\,~~\forall g\in W_{0}^{1,2}(\Omega).
\]
are solvable and eigenvalues can be characterized by the Min-Max Principle (see, for example, \cite{Henr}).

The following result in the case of hyperbolic (conformal) weights was proved in (\cite{BGU15}, Lemma 3.1). In the present article we formulate this lemma in the case of general locally integrable weights.

\begin{lemma}
 \label{lem:TwoWeight} Let $\Omega \subset \mathbb{C}$ be a bounded simply connected domain
and let $w_{1}$, $w_{2}$ be positive a.e. locally integrable functions defined in $\Omega$.
Suppose that there exists a constant $B>0$ such that
\begin{equation}
\iint\limits _{\Omega}|w_{1}(z)-w_{2}(z)||f|^{2}~dxdy\leq B\iint\limits _{\Omega}|\nabla f|^{2}~dxdy,\,\,
\forall f\in W_{0}^{1,2}(\Omega).\label{EqvWW}
\end{equation}

Then for any $n\in\mathbb{N}$
\begin{equation}
|\lambda_{n}[w_{1},\Omega]-\lambda_{n}[w_{2},\Omega]|\leq \frac{B\tilde c_n}{1+B\sqrt{\tilde c_n}}< B\tilde c_n\,,\label{LemEq}
\end{equation}
where
\begin{equation}\label{tilde c_n}
\tilde c_n=\max\{\lambda_{n}^{2}[w_{1},\Omega],\lambda_{n}^{2}[w_{2},\Omega]\} \,.
\end{equation}
\end{lemma}

\begin{proof} By (\ref{MaxPr})
\[
\lambda_{n}[w_{1},\Omega]=\sup\limits _{\substack{f\in M_{n}^{(1)}\\
f\ne0
}
}\frac{\iint\limits _{\Omega}|\nabla f|^{2}~dxdy}{\iint\limits _{\Omega}w_{1}(z)|f|^{2}~dxdy}\,,
\]
 where
\[
M_{n}^{(1)}={\rm span}\,\{\varphi_{1}[w_{1}],...\varphi_{n}[w_{1}]\}.
\]
 Hence by (\ref{EqvWW})
\begin{multline*}
\lambda_{n}[w_{1},\Omega]\geq\sup\limits _{\substack{f\in M_{n}^{(1)}\\
f\ne0
}
}\frac{\iint\limits _{\Omega}|\nabla f|^{2}~dxdy}{\iint\limits _{\Omega}w_{2}(z)|f|^{2}~dxdy+\iint\limits _{\Omega}|w_{1}(z)-w_{2}(z)||f|^{2}~dxdy}\\
\geq\sup\limits _{\substack{f\in M_{n}^{(1)}\\
f\ne0
}
}\frac{\iint\limits_{\Omega}|\nabla f|^{2}~dxdy}{\iint\limits _{\Omega}w_{2}(z)|f|^{2}~dxdy+B\iint\limits _{\Omega}|\nabla f|^{2}~dxdy}\\
=\sup\limits _{\substack{f\in M_{n}^{(1)}\\
f\ne0
}
}\frac{\iint\limits _{\Omega}|\nabla f|^{2}~dxdy}{\iint\limits_{\Omega}w_{2}(z)|f|^{2}~dxdy}\cdot\frac{1}{1+B\frac{\iint\limits _{\Omega}|\nabla f|^{2}~dxdy}{\iint\limits _{\Omega}w_{2}(z)|f|^{2}~dxdy}}\\
\geq\sup\limits _{\substack{f\in M_{n}^{(1)}\\
f\ne0
}
}\frac{\iint\limits _{\Omega}|\nabla f|^{2}~dxdy}{\iint\limits _{\Omega}w_{2}(z)|f|^{2}~dxdy}\cdot\inf\limits _{\substack{f\in M_{n}^{(1)}\\
f\ne0
}
}\frac{1}{1+B\frac{\iint\limits _{\Omega}|\nabla f|^{2}~dxdy}{\iint\limits _{\Omega}w_{2}(z)|f|^{2}~dxdy}}\\
=\sup\limits _{\substack{f\in M_{n}^{(1)}\\
f\ne0
}
}\frac{\iint\limits _{\Omega}|\nabla f|^{2}~dxdy}{\iint\limits _{\Omega}w_{2}(z)|f|^{2}~dxdy}\cdot\frac{1}{1+B\sup\limits _{\substack{f\in M_{n}^{(1)}\\
f\ne0}}\frac{\iint\limits _{\Omega}|\nabla f|^{2}~dxdy}{\iint\limits _{\Omega}w_{2}(z)|f|^{2}~dxdy}}.
\end{multline*}
 Since the function $F(t)={t}/{(1+Bt)}$ is non-decreasing on $[0,\infty)$
and by (\ref{MinMax})
\[
\sup\limits _{\substack{f\in M_{n}^{(1)}\\
f\ne0
}
}\frac{\iint\limits _{\Omega}|\nabla f|^{2}~dxdy}{\iint\limits _{\Omega}w_{2}(z)|f|^{2}~dxdy}\geq\lambda_{n}[w_{2},\Omega],
\]
 it follows that
\[
\lambda_{n}[w_{1},\Omega]\geq\frac{\lambda_{n}[w_{2},\Omega]}{1+B\lambda_{n}[w_{2},\Omega]}=\lambda_{n}[w_{2},\Omega]- \frac{B\lambda_{n}^2[w_{2},\Omega]}{1+B\lambda_{n}[w_{2},\Omega]}.
\]
 Hence
\begin{equation}
\lambda_{n}[w_{1},\Omega]-\lambda_{n}[w_{2},\Omega]\geq-\frac{B\lambda_{n}^2[w_{2},\Omega]}{1+B\lambda_{n}[w_{2},\Omega]}\geq-\frac{B\tilde c_n}{1+B\sqrt{\tilde c_n}}.\label{LemEqR}
\end{equation}
 For similar reasons
\[
\lambda_{n}[w_{2},\Omega]-\lambda_{n}[w_{1},\Omega]\geq -\frac{B\lambda_{n}[w_{1},\Omega]}{1+B\lambda_{n}^2[w_{1},\Omega]}\geq-\frac{B\tilde c_n}{1+B\sqrt{\tilde c_n}}
\]
 or
\begin{equation}
\lambda_{n}[w_{1},\Omega]-\lambda_{n}[w_{2},\Omega]\leq \frac{B\lambda_{n}[w_{1},\Omega]}{1+B\lambda_{n}[w_{1},\Omega]}\le \frac{B\tilde c_n}{1+B\sqrt{\tilde c_n}}.\label{LemEqL}
\end{equation}

Inequalities (\ref{LemEqR}) and (\ref{LemEqL}) imply inequality
(\ref{LemEq}). 
\end{proof}

\vskip 0.2cm

Further we estimate the constant $B$ in Lemma \ref{EqvWW} in terms of an $L^s$-distance between weights.

Firstly we estimate the constant in the corresponding Poincar\'e-Sobolev inequality for a bounded domain $\Omega\subset\mathbb R^2$. 
\begin{theorem}
\label{PoinConst}
Let $\Omega\subset\mathbb R^2$ be a bounded domain and $f \in W^{1,2}_0(\Omega)$. Then 
\begin{equation}\label{InPS}
\|f \mid L^{r}(\Omega)\| \leq A_{r,2}(\Omega) \|\nabla f \mid L^{2}(\Omega)\|, \,\,r \geq 2,
\end{equation}
where
\[
A_{r,2}(\Omega) \leq \inf\limits_{p\in \left(\frac{2r}{r+2},2\right)} 
\left(\frac{p-1}{2-p}\right)^{\frac{p-1}{p}}
\frac{\left(\sqrt{\pi}\cdot\sqrt[p]{2}\right)^{-1}|\Omega|^{\frac{1}{r}}}{\sqrt{\Gamma(2/p) \Gamma(3-2/p)}}.
\]
\end{theorem} 

Note that inequality \eqref{InPS} was proved in \cite{GPU18} in case of the unit disc.

\begin{proof} 
We estimate the constant $A_{r,2}^2(\Omega)$ using the Talenti estimate \cite{Tal76}
\[
\|f \mid L^q(\mathbb R^n)\|\leq A_{p,q}(\mathbb R^n) \|\nabla f \mid L^p(\mathbb R^n)\|,\,\,q=\frac{np}{n-p},\,\, p<n,
\]
where 
\begin{equation*}\label{EsTal}
A_{p,q}(\mathbb R^n)= \frac{1}{\sqrt{\pi}\cdot \sqrt[p]{n}} \left(\frac{p-1}{n-p}\right)^{\frac{p-1}{p}}
\left(\frac{\Gamma(1+n/2) \Gamma(n)}{\Gamma(n/p) \Gamma(1+n-n/p)}\right)^{\frac{1}{n}}.
\end{equation*}

The Talenti estimate can not be applied directly for $p=n=2$. Choose some number $p: 2r/(2+r)<p<2$. By the H\"older inequality with exponents $(2/(2-p), 2/p)$ we have
\begin{multline*}
\biggr(\iint\limits _{\Omega}|\nabla f(x,y)|^{p}\, dxdy\biggr)^{\frac{1}{p}} \leq
\biggr(\iint\limits _{\Omega} dxdy\biggr)^{\frac{2-p}{2p}} \biggr(\iint\limits _{\Omega}|\nabla f(x,y)|^{2}\, dxdy\biggr)^{\frac{1}{2}}\\ = 
|\Omega|^{\frac{2-p}{2p}}\biggr(\iint\limits _{\Omega}|\nabla f(x,y)|^{2}\, dxdy\biggr)^{\frac{1}{2}}.
\end{multline*} 
Because any function $f\in W^{1,p}_0(\Omega)$ can be extended by zero to $\widetilde{f}\in W^{1,p}_0(\mathbb R^n)$, it permit us to apply the Talenti estimate:
\begin{multline*}
\biggr(\iint\limits _{\Omega}|f(x,y)|^{q}\, dxdy\biggr)^{\frac{1}{q}}=\biggr(\iint\limits _{\mathbb R^2}|\widetilde{f}(x,y)|^{q}\, dxdy\biggr)^{\frac{1}{q}}\\
 \leq A_{p,q}(\mathbb R^2)\biggr(\iint\limits _{\mathbb R^2}|\nabla \widetilde{f}(x,y)|^{p}\, dxdy\biggr)^{\frac{1}{p}}=
A_{p,q}(\mathbb R^2)\biggr(\iint\limits _{\Omega}|\nabla f(x,y)|^{p}\, dxdy\biggr)^{\frac{1}{p}},
\end{multline*}
where
\[
A_{p,q}(\mathbb R^2) =  
\frac{1}{\sqrt{\pi}\cdot \sqrt[p]{2}} \left(\frac{p-1}{2-p}\right)^{\frac{p-1}{p}}
\frac{1}{\sqrt{\Gamma(2/p) \Gamma(3-2/p)}}.
\]

Taking into account the H\"older inequality with exponents $(q/(q-r), q/r)$ we get
\begin{multline*}
\biggr(\iint\limits _{\Omega}|f(x,y)|^{r}\, dxdy\biggr)^{\frac{1}{r}} \leq 
\biggr(\iint\limits _{\Omega} dxdy\biggr)^{\frac{q-r}{qr}} \biggr(\iint\limits _{\Omega}|f(x,y)|^{q}\, dxdy\biggr)^{\frac{1}{q}} \\=
|\Omega|^{\frac{q-r}{qr}}\biggr(\iint\limits _{\Omega}|f(x,y)|^{q}\, dxdy\biggr)^{\frac{1}{q}}
\leq |\Omega|^{\frac{q-r}{qr}}A_{p,q}(\mathbb R^2)
\biggr(\iint\limits _{\Omega}|\nabla f(x,y)|^{p}\, dxdy\biggr)^{\frac{1}{p}}
\\
\leq |\Omega|^{\frac{q-r}{qr}} |\Omega|^{\frac{2-p}{2p}}A_{p,q}(\mathbb R^2)
\biggr(\iint\limits _{\Omega}|\nabla f(x,y)|^{2}\, dxdy\biggr)^{\frac{1}{2}}.
\end{multline*}

Since the last inequality holds for any $p\in (2r/(2+r), 2)$ and $q=2p/(2-p)$ we obtain that
\[
\biggr(\iint\limits _{\Omega}|f(x,y)|^{r}\, dxdy\biggr)^{\frac{1}{r}} \leq A_{r,2}(\Omega)
\biggr(\iint\limits _{\Omega}|\nabla f(x,y)|^{2}\, dxdy\biggr)^{\frac{1}{2}},
\]
where
\[
A_{r,2}(\Omega) \leq \inf\limits_{p\in \left(\frac{2r}{r+2},2\right)} 
\left(\frac{p-1}{2-p}\right)^{\frac{p-1}{p}}
\frac{\left(\sqrt{\pi}\cdot\sqrt[p]{2}\right)^{-1}|\Omega|^{\frac{1}{r}}}{\sqrt{\Gamma(2/p) \Gamma(3-2/p)}}.
\]

\end{proof}

\begin{lemma}\label{lem:TwoWeiPol} 
Let   $w_{1}$, $w_{2}$ be positive a.e. locally integrable functions defined in
$\Omega$ such that
\begin{equation}
d_{s}(w_{1},w_{2}):= \|w_1-w_2\mid L^{s}(\Omega)\|<\infty  \label{EqvWWpol}
\end{equation}
 for some $1<s\le\infty$.

Then  inequality $(\ref{EqvWW})$ holds with the constant
\begin{equation}
B=A_{\frac{2s}{s-1},2}^2(\Omega)\,d_{s}(w_{1},w_{2})\,.\label{Lem2Es}
\end{equation}
\end{lemma}

\begin{proof}
By the H\"older inequality and Poincar\'e-Sobolev inequality \eqref{InPS} we get
\begin{multline*}
\iint\limits _{\Omega}|w_{1}(z)-w_{2}(z)||f|^{2}~dxdy
\\
\leq\left(\iint\limits _{\Omega}\left(|w_{1}(z)-w_{2}(z)|\right)^{s}dxdy\right)^{\frac{1}{s}}\left(\iint\limits _{\Omega}|f(z)|^{\frac{2s}{s-1}}dxdy\right)^{\frac{s-1}s}
\\
\leq A_{\frac{2s}{s-1},2}^2(\Omega)\,d_{s}(w_{1},w_{2}) \iint\limits _{\Omega}|\nabla f(z)|^{2}dxdy\,.
\end{multline*}
\end{proof}

By the two previous lemmas we have the following result for variations of the weighted eigenvalues:
\begin{theorem}\label{thm:TwoWW} 
Let  $w_{1}$, $w_{2}$ be positive a.e. locally integrable functions defined in
$\Omega$. Assume that $d_{s}(w_{1},w_{2})<\infty$ for some $s>1$.

Then for any $n\in\mathbb{N}$
\[
|\lambda_{n}[w_{1}]-\lambda_{n}[w_{2}]|\leq \tilde c_{n}A_{\frac{2s}{s-1},2}^2(\Omega) d_{s}(w_{1},w_{2})\,.
\]
\end{theorem}

Now we estimate the quantity $d_{s}(w_{1},w_{2})$ using the $L^2$-norms of weights.
\begin{lemma}\label{L4.5}
Let $w_1$, $w_2$ be a.e. positive in $\Omega$ weights such that $w_k\in L^{\frac{s}{2-s}}$, $k=1,2$, for some $s \in (1, 2]$.
Then 
\[
d_{s}(w_{1},w_{2}) \leq \left(\|w_1\,|\,L^{\frac{2}{2-s}}(\Omega)\|^{\frac{1}{2}} +
\|w_2\,|\,L^{\frac{s}{2-s}}(\Omega)\|^{\frac{1}{2}} \right) \cdot
\|w_1^{\frac{1}{2}}-w_2^{\frac{1}{2}}\,|\,L^{2}(\Omega)\|<\infty.
\]
\end{lemma}

\begin{proof}
By the definitions of $w_1$, $w_2$ and $d_{s}(w_{1},w_{2})$
\begin{multline*}
\left[d_{s}(w_{1},w_{2})\right]^s= 
\iint\limits _{\Omega}\left(|w_{1}(z)-w_{2}(z)|\right)^{s}dxdy \\
=\iint\limits _{\Omega}\left|\sqrt{w_1(z)}+\sqrt{w_2(z)}\right|^{s} \left|\sqrt{w_1(z)}-\sqrt{w_2(z)}\right|^{s}dxdy.
\end{multline*}
Applying to the last integral the H\"older inequality with $q=\frac{2}{s}$
($1\leq q<2$ because $1<s\leq 2$) and $q'=\frac{2}{2-s}$ we get
\begin{multline*}
\left[d_{s}(w_{1},w_{2})\right]^s \leq
\left(\iint\limits _{\Omega}\left|\sqrt{w_1(z)}+\sqrt{w_2(z)}\right|^{\frac{2s}{2-s}}dxdy\right)^{\frac{2-s}{2}} \\
\times \left(
\iint\limits _{\Omega}\left|\sqrt{w_1(z)}-\sqrt{w_2(z)}\right|^{2}dxdy\right)^{\frac{s}{2}}.
\end{multline*}
Now using the triangle inequality we obtain 
\begin{multline*}
d_{s}(w_{1},w_{2}) \leq \|w_1^{\frac{1}{2}}+w_2^{\frac{1}{2}}\,|\,L^{\frac{2s}{2-s}}(\Omega)\| \cdot \|w_1^{\frac{1}{2}}-w_2^{\frac{1}{2}}\,|\,L^{2}(\Omega)\| \\
\leq \left(\|w_1\,|\,L^{\frac{s}{2-s}}(\Omega)\|^{\frac{1}{2}} +
\|w_2\,|\,L^{\frac{s}{2-s}}(\Omega)\|^{\frac{1}{2}} \right) \cdot
\|w_1^{\frac{1}{2}}-w_2^{\frac{1}{2}}\,|\,L^{2}(\Omega)\|.
\end{multline*}
\end{proof}

\section{Eigenvalue problems with quasihyperbolic weights}

In this section we consider weak weighted eigenvalue problem 
\begin{equation}\label{WFWEP}
\iint\limits_\Omega \left\langle \nabla f(z), \nabla \overline{g(z)} \right\rangle dxdy
= \lambda \iint\limits_\Omega h(z)f(z)\overline{g(z)}~dxdy, \,\,\, \forall g\in W_0^{1,2}(\Omega). 
\end{equation}
with the quasihyperbolic (quasiconformal) weight $h=h(z):=|J(z,\varphi^{-1})|$  generated by the $A^{-1}$-quasiconformal mapping $\varphi^{-1}:\Omega\to\widetilde{\Omega}$.

Using Theorem~\ref{L4.1} we prove the weighted Poincar\'e-Sobolev inequality in the bounded simply connected domain $\Omega\subset\mathbb C$ for so-called quasihyperbolic weights which are Jacobians of mappings inverse to $A$-quasiconformal homeomorphisms and obtain solvability of the weighted eigenvalue problem (\ref{WFWEP}).
Denote by  $V^{*}$ the exact constant in the Poincar\'e-Sobolev inequality
\begin{equation*}
\|g\,|\,L^{2}(\widetilde{\Omega})\|\leq V^{*}\|g\,|\,L^{1,2}(\widetilde{\Omega}, A)\|,\,\,\, \forall g\in W_0^{1,2}(\widetilde{\Omega}, A).
\end{equation*}

\begin{theorem}\label{WPIn}
Let $A$ be a matrix satisfies the uniform ellipticity condition~\eqref{UEC} and $\Omega, \widetilde{\Omega} \subset \mathbb C$ be simply connected domains.
Then the weighted embedding operator
\begin{equation}\label{EmbedOper}
i_{\mathbb D}:W_0^{1,2}(\Omega) \hookrightarrow L^2{(\Omega,h)}
\end{equation}
is compact and for any function $f\in W_0^{1,2}(\Omega)$ the inequality 
\[
\|f\,|\,L^{2}(\Omega,h)\|\leq V^{*}\|f\,|\,L^{1,2}(\Omega)\|
\]
holds.
\end{theorem}

\begin{proof} Define the complex dilatation $\mu(z)$ agreed with the matrix $A$ by 
\begin{equation*}
\mu(z)=\frac{a_{22}(z)-a_{11}(z)-2ia_{12}(z)}{\det(I+A(z))}.
\end{equation*} 
Because the matrix $A$ satisfies the uniform ellipticity condition (\ref{UEC}) then 
\begin{equation*}
|\mu(w)|\leq \frac{K-1}{K+1}<1,\,\,\, \text{a.e. in}\,\,\, \widetilde{\Omega},
\end{equation*}
and by \cite{Ahl66} there exists a $\mu$-quasiconformal mapping $\varphi : \widetilde{\Omega} \to \Omega$ agreed with the matrix $A$
i.e. an $A$-quasiconformal mapping. 

Hence by Theorem~\ref{L4.1} the composition operator
\[
\varphi^{*}:L^{1,2}(\Omega)\to L^{1,2}(\widetilde{\Omega}, A),\,\,\, \varphi^{*}(f)= f\circ \varphi
\]
is an isometry. 

Let $f\in L^{1,2}(\Omega) \cap C_0^{\infty}(\Omega)$,
then the composition $g=f\circ \varphi$ belongs to $L^{1,2}(\widetilde{\Omega}, A)$ and because the matrix $A$ satisfies the uniform ellipticity condition (\ref{UEC}) by the Sobolev embedding theorem we can conclude that $g=f\circ \varphi\in W_0^{1,2}(\widetilde{\Omega}, A)$ \cite{M} and the Poincar\'e-Sobolev inequality
\begin{equation}\label{PSIn10}
\|g\,|\,L^{2}(\widetilde{\Omega})\|\leq V^{*}\|g\,|\,L^{1,2}(\widetilde{\Omega}, A)\|
\end{equation}
holds with the exact constant $V^{*}=\lambda_1[A,\widetilde{\Omega}]^{-\frac{1}{2}}$.

Now using the ``transfer" diagram \cite{GG94,GU09} and the change of variable formula for quasiconformal mappings \cite{VGR} we obtain
\begin{multline*}
\|f\,|\,L^{2}(\Omega,h)\|
=\left(\iint \limits_{\Omega} |f(z)|^2h(z)dxdy \right)^{\frac{1}{2}}
=\left(\iint \limits_{\Omega} |f(z)|^2 |J(z,\varphi^{-1})|dxdy \right)^{\frac{1}{2}} \\
=\left(\iint \limits_{\widetilde{\Omega}} |f\circ \varphi(w)|^2dudv \right)^{\frac{1}{2}}
\leq V^{*}\left(\iint \limits_{\widetilde{\Omega}} \left\langle A(w) \nabla (f \circ \varphi(w)), \nabla (f \circ \varphi(w)) \right\rangle dudv \right)^{\frac{1}{2}} \\
=V^{*} \left(\iint \limits_{\Omega} |\nabla f(z)|^2dxdy \right)^{\frac{1}{2}}
=V^{*}\|f\,|\,L^{1,2}(\Omega)\|.
\end{multline*} 
Approximating an arbitrary function  $f \in W_0^{1,2}(\Omega)$ by functions from $C_0^{\infty}(\Omega)$ we have that the weighted Poincar\'e-Sobolev inequality
\[
\|f\,|\,L^{2}(\Omega,h)\| \leq V^{*}\|f\,|\,L^{1,2}(\Omega)\|
\] 
holds for any function $f \in W_0^{1,2}(\Omega)$. 

Further we prove that the embedding operator 
\begin{equation}
i_{\Omega}:W_0^{1,2}(\Omega) \hookrightarrow L^2{(\Omega,h)}
\end{equation}
is compact.  By the same "transfer" diagram \cite{GG94,GU09} this operator can be represented as a composition of three operators: the composition operator $\varphi^{*}_w :W_0^{1,2}(\Omega)\to W_0^{1,2}(\widetilde{\Omega},A)$, the compact embedding operator 
\[i_{\widetilde{\Omega}}:W_0^{1,2}(\widetilde{\Omega},A) \hookrightarrow L^2(\widetilde{\Omega})
\] 
and the composition operator for Lebesgue spaces $(\varphi^{-1})^{*}_l:L^2(\widetilde{\Omega}) \to L^2(\Omega, h)$. 


Firstly we prove that the operator $(\varphi^{-1})^{*}_l$ is an isometry. By the change of variables formula we obtain:
\begin{multline*}
\|f\,|\,L^{2}(\Omega,h)\|
=\left(\iint \limits_{\Omega} |f(z)|^2h(z)~dxdy \right)^{\frac{1}{2}} \\
=\left(\iint \limits_{\Omega} |f(z)|^2 |J(z,\varphi^{-1})|~dxdy \right)^{\frac{1}{2}} 
=\left(\iint \limits_{\widetilde{\Omega}} |f\circ \varphi(w)|^2~dudv \right)^{\frac{1}{2}}
=\|g\,|\,L^{2}(\widetilde{\Omega})\|.
\end{multline*}

Secondly we prove that the composition operator
\[
\varphi_w^{*}:W_0^{1,2}(\Omega)\to W_0^{1,2}(\widetilde{\Omega}, A)
\]
is bounded.

Let $f\in L^{1,2}(\Omega) \cap C_0^{\infty}(\Omega)$. 
Then composition $\varphi^{*}_w(f)= f\circ \varphi \in W_0^{1,2}(\widetilde{\Omega}, A)$.
Thus, given the Poincar\'e-Sobolev inequality \eqref{PSIn10} and the isometry of the
composition operator
\[
\varphi^{*}_w:L^{1,2}(\Omega)\to L^{1,2}(\widetilde{\Omega}, A),
\]
we have
\begin{multline*}
\|\varphi^{*}_w(f)\,|\,L^{2}(\widetilde{\Omega})\|
\leq V^{*}\|\varphi^{*}_w(f)\,|\,L^{1,2}(\widetilde{\Omega}, A)\| \\
= V^{*}\|f\,|\,L^{1,2}(\Omega)\|
\leq V^{*}\|f\,|\,W_0^{1,2}(\Omega)\|.
\end{multline*}
Here $V^{*}$ is the exact constant in the corresponding the Poincar\'e-Sobolev inequality \eqref{PSIn10}. Therefore
\begin{multline*}
\|\varphi^{*}_w(f)\,|\,W_0^{1,2}(\widetilde{\Omega}, A)\|
=\|\varphi^{*}_w(f)\,|\,L^{2}(\widetilde{\Omega})\|+ \|\varphi^{*}_w(f)\,|\,L^{1,2}(\widetilde{\Omega}, A)\| \\
\leq V^{*}\|f\,|\,L^{1,2}(\Omega)\| + \|f\,|\,L^{1,2}(\Omega)\| 
\leq (V^{*}+1)\|f\,|\,W_0^{1,2}(\Omega)\|.
\end{multline*}
Approximating an arbitrary function  $f \in W_0^{1,2}(\Omega)$ by functions in the space $C_0^{\infty}(\Omega)$ and taking into account that quasiconformal mappings possess the Luzin $N^{-1}$-property (preimage of a set measure zero has measure zero) we obtain that the inequality
\[
\|\varphi^{*}_w(f)\,|\,W_0^{1,2}(\widetilde{\Omega}, A)\|
\leq (V^{*}+1)\|f\,|\,W_0^{1,2}(\Omega)\|
\] 
holds for any function $f \in W_0^{1,2}(\Omega)$. 

Hence the embedding operator $i_{\Omega}$ is compact as a composition of the compact operator $i_{\widetilde{\Omega}}$ and bounded operators $\varphi^{*}_w$ and $(\varphi^{-1})^{*}_l$.

\end{proof}

According to Theorem~\ref{WPIn} the weighted embedding operator is compact. By standard arguments we conclude that the spectrum of the weighted eigenvalue problem~\eqref{WFWEP} with quasihyperbolic (quasiconformal)
weights $h$ is discrete and can be written in the form of a non-decreasing sequence
\[
0<\lambda_1[h,\Omega] \leq \lambda_2[h,\Omega] \leq \ldots \leq \lambda_n[h,\Omega] \leq \ldots,
\] 
where each eigenvalue is repeated as many time as its multiplicity (see, for example, \cite{L98}).

By Theorem~\ref{L4.1} the eigenvalue problem (\ref{WFEP})
\begin{equation*}
\iint\limits_{\widetilde{\Omega}} \left\langle \nabla g(w), \nabla \overline{f(w)} \right\rangle dudv
= \lambda \iint\limits_{\widetilde{\Omega}} g(w)\overline{f(w)}~dudv, \,\,\, \forall f\in W_0^{1,2}(\widetilde{\Omega}). 
\end{equation*} is equivalent to the weighted eigenvalue problem (\ref{WFWEP})
\begin{equation*}
\iint\limits_\Omega \left\langle \nabla f(z), \nabla \overline{g(z)} \right\rangle dxdy
= \lambda \iint\limits_\Omega h(z)f(z)\overline{g(z)}~dxdy, \,\,\, \forall g\in W_0^{1,2}(\Omega). 
\end{equation*}
in the domain $\Omega$  
and 
\begin{equation}\label{EqEigenvalues}
\lambda_n[h,\Omega]=\lambda_n[A, \widetilde{\Omega}],\,\,\, n\in \mathbb N.
\end{equation}
For weighted eigenvalues \cite{L98, Henr} we have the following properties:

($i$)  $\lim \limits_{n \to \infty} \lambda_n[h,\Omega]=\infty$.

($ii$) for each $n\in \mathbb N$
\begin{multline}\label{MinMax}
\lambda_n[A, \widetilde{\Omega}]= \inf \limits_{\substack{L \subset W_0^{1,2}(\widetilde{\Omega},A)\\ \dim L=n}}
\sup \limits_{\substack{g\in L \\ g \neq 0}} 
\frac{\iint\limits_{\widetilde{\Omega}} \left\langle A(w) \nabla g, \nabla g\right\rangle dudv}{\iint\limits_{\widetilde{\Omega}} |g|^2dudv} \\
= \inf \limits_{\substack{L \subset W_0^{1,2}(\Omega,h,1)\\ \dim L=n}}
\sup \limits_{\substack{f\in L \\ f \neq 0}} 
\frac{\iint\limits_{\Omega} |\nabla f|^2dxdy}{\iint\limits_{\Omega} |f|^2h(z)dxdy}=\lambda_n[h,\Omega]
\end{multline}
(Min-Max Principle), and
\begin{equation}\label{MaxPr}
\lambda_n[h,\Omega]= \sup \limits_{\substack{f\in M_n \\ f \neq 0}} 
\frac{\iint\limits_{\Omega} |\nabla f|^2dxdy}{\iint\limits_{\Omega} |f|^2h(z)dxdy}
\end{equation} 
where
\[
M_n= \spn \{\varphi_1[h,\Omega], \ldots , \varphi_n[h,\Omega]\}
\]
and $\{\varphi_k[h]\}_{k=1}^{\infty}$ is an orthonormal set of eigenfunctions corresponding to the eigenvalues $\{\lambda_n[h,\Omega]\}_{k=1}^{\infty}$.

($iii$) If $n=1$, then formula~\eqref{MinMax} reduces to 
\begin{multline*}
\lambda_1[A, \widetilde{\Omega}]= \inf \limits_{\substack{g\in W_0^{1,2}(\widetilde{\Omega},A)\\ g \neq 0}} 
\frac{\iint\limits_{\widetilde{\Omega}} \left\langle A(w) \nabla g, \nabla g\right\rangle dudv}{\iint\limits_{\widetilde{\Omega}} |g|^2dudv} \\
= \inf \limits_{\substack{f\in W_0^{1,2}(\Omega,h,1)\\ f \neq 0}}
\frac{\iint\limits_{\Omega} |\nabla f|^2dxdy}{\iint\limits_{\Omega} |f|^2h(z)dxdy}=\lambda_1[h,\Omega].
\end{multline*}
In other words
\begin{equation}\label{FirstEig}
\lambda_1[A,\widetilde{\Omega}]=\lambda_1[h,\Omega]=\frac{1}{(V^{*})^2}
\end{equation}
where $V^{*}$ is the sharp constant in the inequality
\[
\left(\iint\limits_{\Omega}|f|^2h(z)dxdy\right)^{\frac{1}{2}} \leq
V^{*} \left(\iint\limits_{\Omega}|\nabla f|^2dxdy\right)^{\frac{1}{2}}, \,\,\, \forall f\in W_0^{1,2}(\Omega).
\]

\section{On ``\,norm\," $d_s(h)$ for quasiconformal weights $h(z)$}

Let $\widetilde{\Omega}$ be a bounded simply connected domain in $\mathbb C$ and $A\in M^{2 \times 2}(\widetilde{\Omega})$. Assume that there exists an $A$-quasiconformal mapping $\varphi:\widetilde{\Omega} \to \Omega$.
Note that for quasiconformal mappings $\varphi^{-1}:\Omega\to\widetilde{\Omega}$
$$
J_{\varphi^{-1}}(z):=\lim\limits_{r\to 0}\frac{|\varphi^{-1}(B(z,r))|}{|B(z,r)|}=|J(z, \varphi^{-1})|
$$
for almost all $z\in\Omega$. 

In the case $w_1=1$ and $w_2=J_{\varphi^{-1}}$ we define
\begin{equation}
d_s(h):=d_{s}(1,J_{\varphi^{-1}})= \|1-J_{\varphi^{-1}}\mid L^{s}(\Omega)\|<\infty  
\end{equation}

As a consequence of Lemma~\ref{L4.5} in the case of quasihyperbolic weights we have
\begin{corollary}\label{Cor-5.1}
Let $\varphi:\widetilde{\Omega} \to \Omega$ be an $A$-quasiconformal homeomorphism and $h(z)$ be the corresponding quasihyperbolic weight. Assume that $J_{\varphi^{-1}}\in L^{\beta}(\Omega)$ for some $\beta >1$.

Then for $s=\frac{2\beta}{\beta +1}$
\[
d_{s}(h) \leq \left(|\Omega|^{\frac{1}{2\beta}} +
\|J_{\varphi^{-1}}\,|\,L^{\beta}(\Omega)\|^{\frac{1}{2}} \right) \cdot
\|1-J_{\varphi^{-1}}^{\frac{1}{2}}\,|\,L^{2}(\Omega)\|.
\]
\end{corollary}

Now we prove the main result of this paper:

\begin{theorem}\label{Main}
Let $A$ be a matrix satisfies the uniform ellipticity condition \eqref{UEC} and a domain $\widetilde{\Omega}$ be $A$-quasiconformal $\beta$-regular about $\Omega$. 
Then for any $n\in \mathbb N$
\begin{multline*}
|\lambda_n[A, \widetilde{\Omega}]-\lambda_n[A, {\Omega}]| 
\leq c_n A^2_{\frac{4\beta}{\beta -1},2}(\Omega) \\
\times \left(|\Omega|^{\frac{1}{2\beta}} +
\|J_{\varphi^{-1}}\,|\,L^{\beta}(\Omega)\|^{\frac{1}{2}} \right) \cdot
\|1-J_{\varphi^{-1}}^{\frac{1}{2}}\,|\,L^{2}(\Omega)\|,
\end{multline*}
where
$c_n=\max\left\{\lambda_n^2[A, \widetilde{\Omega}], \lambda_n^2[A, \Omega]\right\}$ and
$J_{\varphi^{-1}}$ is a Jacobian of inverse mapping to $A$-quasiconformal homeomorphism $\varphi:\widetilde{\Omega} \to \Omega$. 
\end{theorem}

\begin{proof}
According to Theorem~\ref{thm:TwoWW} we have 
\[
|\lambda_{n}[h,\Omega]-\lambda_{n}[\Omega]|\leq \tilde c_{n}A_{\frac{2s}{s-1},2}^2(\Omega) d_{s}(1,h)\,.
\]
Given Corollary~\ref{Cor-5.1} we get the estimate of quality $d_{s}(1,h)$. For 
$s=\frac{2\beta}{\beta +1}$ we have 
\[
d_{s}(h) \leq \left(|\Omega|^{\frac{1}{2\beta}} +
\|J_{\varphi^{-1}}\,|\,L^{\beta}(\Omega)\|^{\frac{1}{2}} \right) \cdot
\|1-J_{\varphi^{-1}}^{\frac{1}{2}}\,|\,L^{2}(\Omega)\|.
\]
Finally, using equality~\eqref{EqEigenvalues} and taking into account that 
$\frac{2s}{s-1}=\frac{4\beta}{\beta -1}$ for $s=\frac{2\beta}{\beta +1}$, we obtain desired inequality 
\begin{multline*}
|\lambda_n[A, \widetilde{\Omega}]-\lambda_n[A, \Omega]| 
\leq c_n A^2_{\frac{4\beta}{\beta -1},2}(\widetilde{\Omega}) \\
\times \left(|\Omega|^{\frac{1}{2\beta}} +
\|J_{\varphi^{-1}}\,|\,L^{\beta}(\Omega)\|^{\frac{1}{2}} \right) \cdot
\|1-J_{\varphi^{-1}}^{\frac{1}{2}}\,|\,L^{2}(\widetilde{\Omega})\|.
\end{multline*}
\end{proof} 

\vskip 0.3cm

Theorem~\ref{Main} can be precise for Ahlfors-type domains (i.e. quasidiscs).
Recall that $K$-quasidiscs are images of the unit discs $\mathbb D$ under $K$-quasicon\-for\-mal homeomorphisms of the plane $\mathbb C$. The class of quasidiscs includes Lipschitz domains and some fractals domains (snowflakes). The Hausdorff dimension of the the quasidisc's boundary can be any number in [1, 2).

Recall that for any planar $K$-quasiconformal homeomorphism $\psi:\Omega\rightarrow \Omega'$
the following sharp result is known: $J(w,\psi)\in L^p_{\loc}(\Omega)$
for any $1 \leq p<\frac{K}{K-1}$ (\cite{Ast,G81}).

Using the weak inverse H\"older inequality and the sharp estimates of the constants in doubling conditions for measures generated by Jacobians of quasiconformal mappings  \cite{GPU17_2}, we have the following assertion.

\begin{theorem}\label{Quasidisk}
Let $\widetilde{\Omega}\subset\mathbb C$ be a quasidisc and let $\varphi:\widetilde{\Omega} \to \mathbb D$ be an $A$-quasiconformal homeomorphism. 
Assume that $1<\beta <\frac{K}{K-1}$.

Then for any $n\in \mathbb N$
\begin{equation*}
|\lambda_n[A, \widetilde{\Omega}]-\lambda_n[A, \mathbb D]| \\
\leq c_n M_{\beta}(K) 
\|1-J_{\varphi^{-1}}^{\frac{1}{2}}\,|\,L^{2}(\mathbb D)\|,
\end{equation*}
where
$c_n=\max\left\{\lambda_n^2[A, \widetilde{\Omega}], \lambda_n^2[A, \mathbb D]\right\}$ and a sharp constant $M_{\beta}(K)$ depends only on a quasiconformality coefficient K.
\end{theorem}

\begin{remark}
The quantity $M_{\beta}(K)$ in Theorem~\ref{Quasidisk} depends only on a quasiconformality
coefficient K of $\widetilde{\Omega}$:
\begin{multline*}
M(K):=\inf\limits_{1< \beta <\beta^{*}} \Biggl\{
\inf\limits_{p\in \left(\frac{4 \beta}{3\beta -1},2\right)} 
\left(\frac{p-1}{2-p}\right)^{\frac{2(p-1)}{p}}
\frac{\pi^{-\frac{\beta +1}{2\beta}} 4^{-\frac{1}{p}}}{\Gamma(2/p) \Gamma(3-2/p)} \\
\left(\frac{C_{\beta}K \pi^{\frac{1-\beta}{2\beta}}}{2} 
\exp\left\{{\frac{K^2 \pi^2(2+ \pi^2)^2}{4\log3}}\right\}\cdot |\widetilde{\Omega}|^{\frac{1}{2}}+\pi^{\frac{1}{2\beta}}\right) \Biggr\}, \\
C_\beta=\frac{10^{6}}{[(2\beta -1)(1- \nu(\beta))]^{1/2\beta}},
\end{multline*}
where $\beta^{*}=\min{\left(\frac{K}{K-1}, \widetilde{\beta}\right)}$, and $\widetilde{\beta}$ is the unique solution of the equation 
$$\nu(\beta):=10^{8 \beta}\frac{2\beta -2}{2\beta -1}(24\pi^2K^2)^{2\beta}=1.
$$
The function $\nu(\beta)$ is a monotone increasing function. Hence for
any $\beta < \beta^{*}$ the number $(1- \nu(\beta))>0$ and $C_\beta > 0$.
\end{remark}

\begin{proof}
Given that, for $K\geq 1$, $K$-quasidiscs are $A$-quasiconformal $\beta$-regular domains if $1<\beta<\frac{K}{K-1}$. Therefore, by Theorem~\ref{Main} in case $\Omega=\mathbb D$ for $1<\beta<\frac{K}{K-1}$ we have
\begin{multline}\label{Inequal_1}
|\lambda_n[A, \widetilde{\Omega}]-\lambda_n[A, \mathbb D]|
\leq c_n A^2_{\frac{4\beta}{\beta -1},2}(\mathbb D) \\
\times \left(\pi^{\frac{1}{2\beta}} +
\|J_{\varphi^{-1}}\,|\,L^{\beta}(\mathbb D)\|^{\frac{1}{2}} \right) \cdot
\|1-J_{\varphi^{-1}}^{\frac{1}{2}}\,|\,L^{2}(\mathbb D)\|.
\end{multline}
Now we estimate the quantity $\|J_{\varphi^{-1}}\,|\,L^{\beta}(\mathbb D)\|$. Taking into account (Corollary~5.2, \cite{GPU2019}) we get
\begin{multline}\label{Inequal_2}
\|J_{\varphi^{-1}}\,|\,L^{\beta}(\mathbb D)\| = 
\left(\iint\limits_{\mathbb D} |J(z,\varphi^{-1})|^{\beta}~dxdy \right)^{\frac{1}{\beta}} \\
\leq \frac{C^2_{\beta} K^2 \pi^{\frac{1-\beta}{\beta}}}{4} \exp\left\{{\frac{K^2 \pi^2(2+ \pi^2)^2}{2\log3}}\right\} \cdot |\widetilde{\Omega}|. 
\end{multline}
Combining inequality \eqref{Inequal_1} with inequality \eqref{Inequal_2} and given that
\[
A_{\frac{4\beta}{\beta -1},2}^2(\mathbb D) \leq \inf\limits_{p\in \left(\frac{4 \beta}{3\beta -1},2\right)} 
\left(\frac{p-1}{2-p}\right)^{\frac{2(p-1)}{p}}
\frac{\pi^{-\frac{\beta +1}{2\beta}} 4^{-\frac{1}{p}}}{\Gamma(2/p) \Gamma(3-2/p)}
\]
we obtain the required result.
\end{proof}

\section{Isospectral operators}

Let us remind that two linear operators are called isospectral if they have the same spectrum. It is known that there are distinct domains such that all the eigenvalues of the linear operator (in the case of the Laplace operator, see, for instance, \cite{BCDS}) coincide. For this reason, these are called isospectral domains. 

Let $\varphi:\widetilde{\Omega} \to \Omega$ be $A$-quasiconformal mappings for which $|J(w,\varphi)|=1$ for almost all $w \in \widetilde{\Omega}$. In this case quasihyperbolic weights $h(z)=|J(z, \varphi^{-1})|=1$ for almost all $z \in \Omega$. Hence, formula~\eqref{MinMax} is written as 
\begin{multline}\label{MinMax-L}
\lambda_n[A, \Omega]= \inf \limits_{\substack{L \subset W_0^{1,2}(\widetilde{\Omega},A)\\ \dim L=n}}
\sup \limits_{\substack{g\in L \\ g \neq 0}} 
\frac{\iint\limits_{\widetilde{\Omega}} \left\langle A(w) \nabla g, \nabla g\right\rangle dudv}{\iint\limits_{\widetilde{\Omega}} |g|^2dudv} \\
= \inf \limits_{\substack{L \subset W_0^{1,2}(\Omega)\\ \dim L=n}}
\sup \limits_{\substack{f\in L \\ f \neq 0}} 
\frac{\iint\limits_{\Omega} |\nabla f|^2dxdy}{\iint\limits_{\Omega} |f|^2dxdy}=\lambda_n[\Omega].
\end{multline}
From here we conclude that the Dirichlet eigenvalues for the elliptic operator in divergence form $-\textrm{div} [A(w) \nabla g(w)]$ in a domain $\widetilde{\Omega}$ are equal to the Dirichlet eigenvalues for the Laplace operator in a domain $\Omega$.

\vskip 0.2cm 
Now we consider some examples of isospectral operators for $\Omega=\mathbb D$.
 
\vskip 0.3cm
$\mathbf{Example \, A.}$ The homeomorphism 
$$
\varphi(w)= w e^{2i \log |w|},\,\, \varphi(0)=0, \quad w=u+iv,
$$
is $A$-quasiconformal and maps the unit disc $\mathbb D$ 
onto itself, transform radial lines into spiral infinitely winding around the origin \cite{GM01}.
The mapping $\varphi$ satisfies the Beltrami equation with
\[
\mu(w)=\frac{\varphi_{\overline{w}}}{\varphi_{w}}=\frac{1+i}{2}\frac{w}{\overline{w}}
\]
and the Jacobian $J(w,\varphi)=|\varphi_{w}|^2-|\varphi_{\overline{w}}|^2=1$.
It is not difficult to verify that $\mu$ induces, by formula \eqref{Matrix-F}, the matrix function $A(w)$, which in the polar coordinates $w=\rho e^{i \theta}$ has the form 
$$
A=\begin{pmatrix} 3-2\sqrt{2} \cos(2\theta + \pi/4) & -2\sqrt{2} \sin(2\theta + \pi/4) \\ -2\sqrt{2} \sin(2\theta + \pi/4) & 3+2\sqrt{2} \cos(2\theta + \pi/4) \end{pmatrix}.
$$ 
Since $|J(w,\varphi)|=|J(z,\varphi^{-1})|^{-1}$ then $h(z)=|J(z, \varphi^{-1})|=1$. Hence, equality~\eqref{MinMax-L} implies $\lambda_n[A, \mathbb D]=\lambda_n[\mathbb D],\,\,\, \forall \,\, n\in \mathbb N$. 

\vskip 0.3cm

$\mathbf{Example \, B.}$ The homeomorphism 
\[
\varphi(w)= \sqrt{a^2+1}w-a \overline{w}, \quad w=u+iv, \quad a\geq 0,
\]
is a $A$-quasiconformal and maps the interior of ellipse
$$
\Omega_e= \left\{(u,v) \in \mathbb R^2: \frac{u^2}{(\sqrt{a^2+1}+a)^2}+\frac{v^2}{(\sqrt{a^2+1}-a)^2}=1\right\}
$$
onto the unit disc $\mathbb D.$ The mapping $\varphi$ satisfies the Beltrami equation with
\[
\mu(w)=\frac{\varphi_{\overline{w}}}{\varphi_{w}}=-\frac{a}{\sqrt{a^2+1}}
\]
and the Jacobian $J(w,\varphi)=|\varphi_{w}|^2-|\varphi_{\overline{w}}|^2=1$.
It is easy to verify that $\mu$ induces, by formula \eqref{Matrix-F}, the matrix function $A(w)$ form 
$$
A=\begin{pmatrix} (\sqrt{a^2+1}+a)^2 & 0 \\ 0 &  (\sqrt{a^2+1}-a)^2 \end{pmatrix}.
$$ 
Since $|J(w,\varphi)|=|J(z,\varphi^{-1})|^{-1}$ then $h(z)=|J(z, \varphi^{-1})|=1$. Hence, equality~\eqref{MinMax-L} implies $\lambda_n[A,\Omega_e]=\lambda_n[\mathbb D],\,\,\, \forall \,\, n\in \mathbb N$.

\vskip 0.3cm

$\mathbf{Example \, C.}$ The homeomorphism 
\[
\varphi(w)= \frac{w^{\frac{3}{2}}}{\sqrt{2} \cdot \overline{w}^{\frac{1}{2}}}-1,\,\, \varphi(0)=-1, \quad w=u+iv,
\]
is $A$-quasiconformal and maps the interior of the ``rose petal"
$$
\Omega_p:=\left\{(\rho, \theta) \in \mathbb R^2:\rho=2\sqrt{2}\cos(2 \theta), \quad -\frac{\pi}{4} \leq \theta \leq \frac{\pi}{4}\right\}
$$
onto the unit disc $\mathbb D$. 
The mapping $\varphi$ satisfies the Beltrami equation with
\[
\mu(w)=\frac{\varphi_{\overline{w}}}{\varphi_{w}}=-\frac{1}{3}\frac{w}{\overline{w}}
\]
and the Jacobian $J(w,\varphi)=|\varphi_{w}|^2-|\varphi_{\overline{w}}|^2=1$. 
We see that $\mu$ induces, by formula \eqref{Matrix-F}, the matrix function $A(w)$, which in the polar coordinates $w=\rho e^{i \theta}$ has the form 
$$
A=\begin{pmatrix} 2\cos^2{\theta}+1/2\sin^2{\theta} & 3/4\sin{2\theta} \\ 3/4\sin{2\theta} & 1/2\cos^2{\theta}+2\sin^2{\theta} \end{pmatrix}.
$$
Since $|J(w,\varphi)|=|J(z,\varphi^{-1})|^{-1}$ then $h(z)=|J(z, \varphi^{-1})|=1$. Hence, equality~\eqref{MinMax-L} implies $\lambda_n[A,\Omega_p]=\lambda_n[\mathbb D],\,\,\, \forall \,\, n\in \mathbb N$.

\vskip 0.2cm

{\bf Acknoledgement.}
The first author was supported by the United States-Israel Binational Science
Foundation (BSF Grant No. 2014055).
The second author was partly supported by the Ministry of Education and Science of the Russian Federation 
in the framework of the research Project No. 2.3208.2017/4.6, by RFBR Grant No. 18-31-00011.

\vskip 0.3cm

Department of Mathematics, Ben-Gurion University of the Negev, P.O.Box 653, Beer Sheva, 8410501, Israel 
 
\emph{E-mail address:} \email{vladimir@math.bgu.ac.il} \\           
       
Division for Mathematics and Computer Sciences, Tomsk Polytechnic University,
634050 Tomsk, Lenin Ave. 30, Russia; International Laboratory SSP \& QF, Tomsk
State University, 634050 Tomsk, Lenin Ave. 36, Russia

\emph{Current address:} Department of Mathematics, Ben-Gurion University of the Negev, P.O.Box 653, 
Beer Sheva, 8410501, Israel  
							
\emph{E-mail address:} \email{vpchelintsev@vtomske.ru}   \\
			  
Department of Mathematics, Ben-Gurion University of the Negev, P.O.Box 653, Beer Sheva, 8410501, Israel 
							
\emph{E-mail address:} \email{ukhlov@math.bgu.ac.il

\end{document}